\documentclass[11 pt]{amsart}

\usepackage[margin=2.5cm]{geometry}
\usepackage{amsfonts, enumerate,mathtools}
\usepackage{amssymb}
\usepackage{amsmath}
\usepackage{latexsym}
\usepackage{mathrsfs}
\usepackage{amsthm}
\usepackage{verbatim}
\usepackage{amscd}
\usepackage{hyperref}
\usepackage{setspace}

\usepackage{comment}
\usepackage{framed}
\usepackage{color}
\definecolor{shadecolor}{gray}{0.875}

\newtheorem{thrm}{Theorem}[section]

\newtheorem{lem}[thrm]{Lemma}
\newtheorem{cor}[thrm]{Corollary}
\newtheorem{prop}[thrm]{Proposition}

\theoremstyle{definition}
\newtheorem{defn}[thrm]{Definition}

\newtheorem{exmple}[thrm]{Example}

\newtheorem{rmk}[thrm]{Remark}

\DeclareMathOperator{\Eff}{\overline{Eff}}

\input xy
\xyoption{all}

\begin{document}

\title{Kernels of numerical pushforwards}

\author{Mihai Fulger}
\address{Department of Mathematics, Princeton University\\
Princeton, NJ \, \, 08544}
\address{Institute of Mathematics of the Romanian Academy, P. O. Box 1-764, RO-014700,
Bucharest, Romania}
\email{afulger@princeton.edu}

\author{Brian Lehmann}
\thanks{The second author is supported by NSF Award 1004363.}
\address{Department of Mathematics, Boston College  \\
Chestnut Hill, MA \, \, 02467}
\email{lehmannb@bc.edu}

\begin{abstract}
Let $\pi: X \to Y$ be a morphism of projective varieties and consider the pushforward map $\pi_{*}: N_{k}(X) \to N_{k}(Y)$ of numerical cycle classes.  We show that when the Chow groups of points of the fibers are as simple as they can be, then the kernel of $\pi_{*}$ is spanned by $k$-cycles contracted by $\pi$.
\end{abstract}


\maketitle

\section{Introduction}

Let $\pi: X \to Y$ be a morphism of projective varieties over an algebraically closed field.  The pushforward of cycles induces a map $\pi_{*}: N_{k}(X) \to N_{k}(Y)$ on numerical groups with $\mathbb{R}$-coeffficients. One would like to understand how $\ker(\pi_{*})$ reflects the geometry of the map $\pi$.  In the special case when $\alpha \in N_{k}(X)$ is the class of a closed subvariety $Z$, then $\alpha$ lies in the kernel of $\pi_{*}$ precisely when $\dim \pi(Z) < \dim(Z)$. A similar statement holds when $\alpha$ is the class of an effective cycle. However, the geometry of arbitrary elements of $\ker(\pi_{*})$ is more subtle, e.g. Example \ref{ex:abeliansurface}. We are interested in identifying classes of morphisms $\pi$ for which the kernel of $\pi_*$ reflects the ``geometry'' of the morphism in the simplest way: it is spanned by effective cycles contracted by $\pi$.


Perhaps the easiest example where this property fails is the following.  

\begin{exmple} \label{ex:abeliansurface}
Let $E$ be an elliptic curve without complex multiplication and define $X = E \times E$.  Then the Neron-Severi space $N_{1}(X)$ is $3$-dimensional with a basis given by the classes $F_{1}$, $F_{2}$ of fibers of the two projections and by the diagonal class $\Delta$.  The pseudo-effective cone $\Eff_{1}(X)$ is a round cone.

Consider the first projection $\pi: X \to E$.  The curves contracted by $\pi$ are all numerically equivalent with class $F_{1}$.  However, the kernel of $\pi$ is two-dimensional: it is generated by $F_{1}$ and by $\Delta - F_{2}$.  In particular, the kernel is \emph{not} spanned by the subvarieties contracted by $\pi$.
\end{exmple}


We expect that the failure 
in the previous example is caused by the fibers having large Chow groups.  Our main result supports this.

\begin{thrm}\label{firstgkthrm}
Let $\pi: X \to Y$ be a dominant morphism of projective varieties over an uncountable algebraically closed field, with $Y$ smooth.  Suppose that \emph{every} fiber $F$ (over a closed point) of $\pi$ satisfies $\dim_{\mathbb Q}(A_{0}(F)_{\mathbb Q}) =1$.  Then the kernel of $\pi$ is spanned by effective $\pi$-contracted cycles.
\end{thrm}

The corresponding statement for rational equivalence is well-established (see \cite{bs83}). 
The main point is to translate between rational and numerical equivalence for the map $\pi$.  The hypothesis is satisfied by any rationally chain connected variety.  More generally, the varieties satisfying the hypothesis are regulated by Bloch's Conjecture; for example, for a smooth surface $S$ over $\mathbb{C}$ the condition $\dim_{\mathbb Q}(A_{0}(S)_{\mathbb Q}) =1$ is expected to be equivalent to $\dim H^{0}(S,K_{S}) = 0$.  

A few particular cases are worth special note:

\begin{prop}[cf. \ref{swconjforbirationalsmooth}]
Let $\pi: X \to Y$ be a birational morphism over $\mathbb{C}$ with $Y$ smooth.  Then the kernel of $\pi$ is spanned by effective $\pi$-contracted cycles.
\end{prop}

\begin{cor}[cf. \ref{cor:rccGK}]
Let $\pi:X\to Y$ be a morphism of projective varieties over $\mathbb C$ with $Y$ smooth and with rationally chain connected general fiber.  Then the kernel of $\pi$ is spanned by effective $\pi$-contracted cycles.
\end{cor}

A related result for curve classes appears in \cite[Corollary 12.1.5.1]{KM92}.

In a different direction, \cite{djv13} investigates (for an arbitrary morphism $\pi$) which $\pi_*$-contracted classes are actually linear combinations of classes of $\pi$-contracted subvarieties.
The main conjecture of \cite{djv13} is that any $\pi_*$-contracted \emph{pseudo}-effective class satisfies this linear combination property.  
In a forthcoming paper we use constructions from \cite{fl14} to 
prove new cases of their conjecture.



\section{Background and conventions} \label{numeqsection}

\noindent By \textit{variety} we mean a reduced, irreducible, separated scheme of finite type over an algebraically closed field of arbitrary characteristic. Unless otherwise stated, $\pi:X\to Y$ denotes a morphism of projective varieties. We say that a $k$-cycle $Z$ is a $\mathbb{Z}$, $\mathbb{Q}$, or $\mathbb{R}$-$k$-cycle depending on whether the coefficients lie in $\mathbb{Z}$, $\mathbb{Q}$, or $\mathbb{R}$. A cycle that is a positive combination of subvarieties is called \emph{effective}. The group of all $k$-$\mathbb{Z}$-cycles is denoted by $Z_{k}(X)$ and the Chow group of $k$-$\mathbb{Z}$-cycles up to rational equivalence (\cite[\S1]{fulton84}) is denoted by $A_{k}(X)$. A closed subscheme $Y\subset X$ determines a \emph{fundamental $\mathbb Z$-cycle} $[Y]$ (cf. \cite[\S1.5]{fulton84}).

Let $\pi: X \to Y$ be a morphism of projective varieties and let $Z$ be a $k$-cycle on $X$.  We say that $Z$ is $\pi$-\emph{contracted} if for every component $Z_{i}$ of $Z$ we have $\dim(\pi(Z_{i}))<\dim(Z_{i})$.

We let $N_{k}(X)_{\mathbb{Z}}$ denote the quotient of $Z_{k}(X)$ by the relation of \emph{numerical equivalence} as in \cite[Chapter 19]{fulton84}.  $N_{k}(X)_{\mathbb{Z}}$ is a lattice inside the \emph{numerical space} $N_k(X):=N_k(X)_{\mathbb Z}\otimes_{\mathbb Z}\mathbb R.$
If $Z$ is a $k$-cycle with $\mathbb{R}$-coefficients, its class in $N_k(X)$ is denoted $[Z]$. The real N\' eron-Severi space of numerical classes of Cartier divisors is denoted $N^1(X)$. Its elements act as first Chern classes on the numerical groups $N_k(X)$. 

For the rest of the paper, the term cycle will always refer to a cycle with $\mathbb{Z}$-coefficients, and a numerical class will always refer to a class with $\mathbb{R}$-coefficients, unless otherwise qualified.

\section{The GK property} \label{gksection}

\begin{defn}A morphism of projective varieties $\pi:X\to Y$ satisfies the \emph{Geometric Kernel} (or GK) property if every class $\alpha \in N_{k}(X)$ with $\pi_*\alpha=0$ lies in the vector space generated by the $k$-dimensional subvarieties which are contracted by $\pi$.
\end{defn}


We begin with a few easy observations and important particular cases. Among them we will see that the GK property holds over $\mathbb C$ for birational maps $\pi$ with a smooth base.

\subsection{Compositions of morphisms}

Note that if $\pi: X \to Y$ is a dominant morphism of projective varieties then $\pi_{*}: N_{k}(X) \to N_{k}(Y)$ is surjective, since any $k$-dimensional subvariety of $Y$ is mapped onto by a $k$-dimensional subvariety of $X$.

\begin{lem}\label{GKcompose} The ($k$-th) GK property is stable under composition
of dominant projective morphisms.
\end{lem}

\begin{proof}Let $\pi_2:X_2\to X_1$ and $\pi_1:X_1\to X_0$ be dominant morphisms of projective varieties satisfying the GK property. For $i\in\{1,2\}$, choose a finite set $\Sigma_i$ of $k$-dimensional subvarieties of $X_i$ whose classes generate $\ker\pi_{i*}$. Let $\Sigma'_{1}$ denote a set of $k$-dimensional subvarieties of $X_{2}$ such that each subvariety in $\Sigma_{1}$ is dominated under $\pi_{2}$ by a subvariety in $\Sigma'_{1}$.  Then the elements of $\Sigma'_1\cup \Sigma_2$ are contracted by $\pi_1\circ\pi_2$, and their classes span $\ker(\pi_1\circ\pi_2)_*$.   
\end{proof}

\begin{lem}\label{GKdescend} If $\pi_1:X_1\to X_0$ and $\pi_2:X_2\to X_1$
are dominant morphisms such that $\pi_1\circ\pi_2$ satisfies the GK property,
then $\pi_1$ satisfies the GK property.
\end{lem}
\begin{proof}There is a surjection $\pi_{2*}: \ker(\pi_{1*} \circ \pi_{2*}) \to \ker(\pi_{1*})$, and the desired conclusion follows. \end{proof}

\subsection{The GK property for projective bundles and birational maps}

\begin{prop}\label{weakprojbundle}Projective bundle maps satisfy the GK property.\end{prop}

\begin{proof}
Let $X=\mathbb P_Y(E)$ for some vector bundle $E$ of rank $r$ on $Y$, and denote by $\pi:X\to Y$ the bundle map. 
Let $\xi=[c_1(\mathcal O_{\mathbb P(E)}(1))]\in N^1(X)$. 
Up to twisting $E$, we can assume that $\xi$ is ample.

Suppose $\pi_{*}\alpha = 0$ for some $\alpha \in N_{k}(X)$.  Using \cite[Theorem 3.3.(b)]{fulton84}, write
$$\alpha=\sum_{i=0}^{r-1}\xi^{r-1-i}{\cdotp}\pi^*a_{k-i}$$ 
for some classes $a_{k-i}\in N_{k-i}(Y)$. Since $\pi_*\xi^{r-1}=[Y]\in N_{\dim Y}(Y)$, the condition $\pi_*\alpha=0$ is equivalent to $a_{k}=0$. Thus, it suffices to see that for each $i>0$ the term $\xi^{r-1-i} \cdotp \pi^{*}a_{k-i}$ is represented by a linear combination of cycles contracted by $\pi$.

Let $m$ be a sufficiently positive integer so that $m\xi$ is very ample.  For any cycle $Z$, the numerical class $\xi \cdot [Z]$ is represented by $\frac{1}{m}H|_{Z}$ where $H$ is a general element of $|m\xi|$.  Using this construction repeatedly, we see that $\xi^{r-1-i} \cdotp \pi^{*}a_{k-i}$ is represented by a cycle with positive relative dimension over $Y$ for any $i>0$ and any class $a \in N_{k-i}(Y)$.
\end{proof}

\begin{rmk}
Essentially the same argument shows that Grassmann  bundle maps satisfy the GK property.
\end{rmk}


\begin{lem}\label{GKblow-up} Let $\pi:X\to Y$ be the blow-up of a smooth variety $Y$ along a smooth center. Then $\pi$ satisfies the GK property.\end{lem}

\begin{proof}Let $Z$ be the blow-up center so that we have a cartesian diagram:
$$\xymatrix{\widetilde Z\ar[r]^{\jmath}\ar[d]_{f}& X\ar[d]^{\pi}\\
Z\ar[r]_{\imath}& Y}$$
with $f:\widetilde Z\to Z$ the projective bundle $\mathbb P(N^{\vee}_ZY)$. By \cite[Proposition 6.7]{fulton84}, 
any cycle $\alpha$ on $X$ is of the form $\jmath_*\tilde z+\pi^*y$. The condition $\pi_*\alpha=0$ implies $y=-\imath_*f_*\tilde z$. By \cite[Proposition 6.7.(a) and Example 3.3.3]{fulton84}, we have $\alpha=\jmath_*\beta$ with $f_*\beta=0$. Proposition \ref{weakprojbundle} completes the proof. \end{proof}

\begin{rmk}The pullback map $\pi^*$ appearing in the proof of Proposition \ref{weakprojbundle}
is well-defined for Chow groups and respects numerical equivalence because $\pi$ is smooth (cf. \cite[Example 19.2.3]{fulton84}).
This is of slightly different flavor from the pullback $\pi^*$ used in the proof of Lemma \ref{GKblow-up}.
Here, $\pi^*$ is well-defined on Chow groups and respects numerical equivalence because $Y$ is smooth (cf. \cite[Example 19.1.6]{fulton84}).

We caution the reader that it is not clear that the flat (non-smooth) pullback for Chow groups from \cite[\S1.7]{fulton84} respects numerical equivalence. We will not be concerned with this difficulty here. 
\end{rmk}

\begin{prop} \label{swconjforbirationalsmooth}
Suppose that $\pi: X \to Y$ is a birational morphism of varieties over $\mathbb{C}$ with $Y$ smooth.  Then $\pi$ satisfies the GK property.
\end{prop}

\noindent Note that the case of divisors is well-known; in fact, the Negativity of Contraction lemma proves the GK property for divisor classes when we have a birational map in arbitrary characteristic with a $\mathbb{Q}$-factorial base.

\begin{proof}
The birational map $\pi$ is dominated by a composition of blow-ups of smooth varieties along smooth centers. The result follows by Lemma \ref{GKcompose}, Lemma \ref{GKdescend} and Lemma \ref{GKblow-up}. 
\end{proof}

\begin{exmple}The smoothness of $Y$ in the proposition is necessary. For example if $Z$ is the projective cone over $E\times E$, where $E$ is an elliptic curve, and $\tau:X\to Z$ is the blow-up of the vertex with exceptional divisor $D\simeq E\times E$, then $X$ has a structure of a projective bundle over $E\times E$ with bundle map $f$. Let $\pi:X\to Y$ be the contraction of a fiber of the first projection $D=E\times E\to E$. Then $$\pi_*(D\cdot f^*(\Delta-F_2))=0,$$ where $\Delta$ is the class of the diagonal on $E\times E$ and $F_2$ is a fiber of the second projection. However the only effective curve contracted by $\pi$ is $D\cdot f^*F_1$, and $D\cdot f^*(\Delta-F_2)$ is not a multiple of it.  
\qed
\end{exmple}

\subsection{The GK property and Chow groups of points}

We first look to understand the relative behavior of Chow groups under a morphism $\pi$. For this we use a variant of the Bloch--Srinivas theorem presented in \cite[Theorem 10.19]{voisin07}.

\begin{thrm}
Let $\pi: X \to Y$ be a projective map of quasi-projective schemes of finite type over an uncountable algebraically closed field with $Y$ integral.  Let $Z$ be a $k$-cycle on $X$ with rational equivalence class $\tau$.  Suppose there is a subscheme $X'$ of $X$ such that for a general smooth closed point $y \in Y$, the restriction $\tau|_{X_{y}}$ vanishes in $A_{*}(X_{y} - X'_{y})_{\mathbb Q}$.  

Then there is a positive integer $m$, a cycle $Z_{1}$ in $X'$, and a cycle $Z_{2}$ whose support does not dominate $Y$ such that $mZ \sim Z_{1} + Z_{2}$ in $A_{k}(X)$.
\end{thrm}

Since this statement is undoubtedly well-known to experts, we only give a brief verification.

\begin{proof}Let $\overline X$ be a projective closure of $X$, and let $H$ be a very ample divisor on $\overline X$. By $|rH|$ we denote the complete linear series on $\overline X$. If $W$ is a projective subvariety of $X$, then $|rH||_W$ naturally surjects onto $|rH|_W|$ for all $r\gg 0$.

Most of the proof of \cite[Theorem 10.19]{voisin07}, whose notation we retain, goes through. 
We list the changes:
\begin{itemize}
\item The condition $\tau|_{X_y}=0\in A_*(X_y-X'_y)_{\mathbb Q}$ allows torsion $\mathbb Z$-classes. Consequently the relative Hilbert schemes $H_i$ used in the proof of \cite{voisin07} are now defined to parameterize rational equivalence relations between $sZ_y$ and cycles in $X'_y$ for all positive integers $s$.
\item  The resolution of the subvariety $H$ mapping generically finitely onto $Y$
is replaced by a nonsingular alteration (\cite{dejong96}).
\item The divisors ${\rm div}(\phi_{y,l})$ on $W_{y,l}$ are differences of elements of
$|rH|_{W_{y,l}}|$, hence they can be seen as differences of restrictions of elements
of $|rH|$ and parameterized accordingly. Then the divisors $\mathcal D_l$ can be constructed
directly as differences of restrictions to $\mathcal W_l$ of divisors on $\mathbb P(|rH|)\times H$ for some $r\gg 0$, in particular they are Cartier.
\item As in the proof of \cite[Lemma 10.22]{voisin07}, one uses Grauert's theorem to show that the Cartier divisors $\mathcal D_l$ are pullbacks from $Y$, at least over an open subset of $Y$.
\end{itemize}  
\end{proof}

We then immediately obtain:

\begin{cor} \label{cor:makingvertical}
Let $\pi: X \to Y$ be a projective map of quasi-projective schemes of finite type over an uncountable algebraically closed field, with $Y$ integral of dimension $k$.  Suppose that for a fiber $X_{y}$ over a general closed point $y$ we have $\dim_{\mathbb Q}(A_{0}(F)_{\mathbb Q}) =1$.  Then if $Z$ is a $k$-cycle on $X$ such that $\pi_{*}Z \sim_{\mathbb Q} 0$, there is a positive integer $m$ and a $\pi$-contracted cycle $Z'$ such that $mZ \sim_{\mathbb Z} Z'$.
\end{cor}

\begin{proof}
Using the restriction exact sequence for Chow groups, we may shrink $Y$ to assume that it is smooth.

If every component of $Z$ is contracted by $\pi$, set $Z' = Z$.  Otherwise, $Z$ must have some component $Z_{1}$ that surjects onto $Y$.  Since a general fiber $X_{y}$ has ${\rm rank}(A_{0}(X_{y}))=1$, the restriction of $Z$ to $A_{0}(X_{y} - Z_{1,y})_{\mathbb Q}$ must vanish.  Thus, up to adding some $\pi$-contracted cycle, a multiple of $Z$ is rationally equivalent to a multiple of $Z_{1}$.  But of course this multiple must be $0$ since it pushes forward to $0$.
\end{proof}

\begin{rmk} \label{domremark} Consider a dominant morphism of projective varieties $\pi:X\to Y$ with $Y$ smooth. Then there is an induced surjective morphism $$A_k(X)_{\mathbb Q}\cap\ker\pi^{\rm Chow_{\mathbb Q}}_*\twoheadrightarrow N_k(X)_{\mathbb Q}\cap\ker\pi^{\rm num_{\mathbb Q}}_*.$$ 
(If $\pi$ has relative dimension $e$ and if $h$ is an ample class on $X$ normalized so that $\pi_*(h^e\cap[X])=[Y]$, then $\pi_*(h^e\cdot\pi^*\alpha)=\alpha$ for any $\alpha\in A_k(Y)$. This and \cite[Example 19.1.6]{fulton84} imply that $\pi_*$ maps numerically trivial Chow classes on $X$ surjectively onto numerically trivial Chow classes on $Y$.  Conclude by the Snake Lemma.)\end{rmk}


\begin{lem} \label{lem:replacingvanishingpushforwardcycles}
Let $\pi: X \to Y$ be a surjective map of projective varieties with $Y$ smooth, and let $Z$ be a $k$-$\mathbb{R}$-cycle on $X$ such that $\pi_{*}[Z] = 0\in N_k(Y)$.  Then there is some $k$-$\mathbb{R}$-cycle $Z' \equiv Z$ such that $\pi_{*}Z' = 0$ as cycles.
\end{lem}

\begin{proof}Since $\pi_*$ is defined over $\mathbb Z$, we can replace $\mathbb R$ with $\mathbb Q$ throughout. By Remark \ref{domremark}, there exists a cycle $Z_1\equiv Z$ such that $\pi_*Z_1\in{\rm Rat}_k(Y)$, where ${\rm Rat}_k(Y)$ denotes the group of $k$-$\mathbb Q$-cycles rationally equivalent to zero on $Y$. It is then enough to show that $\pi_*$ maps
${\rm Rat}_k(X)$ onto ${\rm Rat}_k(Y)$. By \cite[\S1.3]{fulton84}, if $A\in{\rm Rat}_k(Y)$, then there exist finitely many $k+1$-dimensional subvarieties $V_i$ of $Y$ and elements $r_i\in R(V_i)^*$ such that $A=\sum_i{\rm div}(r_i)$. 
Let $W_i$ be $k+1$-dimensional subvarieties of $X$ mapping generically finitely onto $V_i$, and denote $d_i=\deg(W_i/V_i)$. Then $A'=\sum_i\frac 1{d_i}{\rm div}(r_i)\in{\rm Rat}_kX$ satisfies $\pi_*A'=A$ by \cite[Proposition 1.4.(b)]{fulton84}.
\end{proof}

\section{The main result}

\subsection*{Proof of Theorem \ref{firstgkthrm}}
It suffices to consider the case when $k \leq \dim Y$.  Let $Z$ be any $k$-cycle such that $\pi_{*}[Z] = 0$.  Apply Lemma \ref{lem:replacingvanishingpushforwardcycles} to replace $Z$ by the numerically equivalent cycle $Z'$ as in the statement.  Thus $Z'$ is a sum of cycles $Z_{i}$ such that each $Z_{i}$ satisfies $\pi_{*}Z_{i} = 0$ and either:
\begin{itemize}
\item $|Z_{i}|$ is $\pi$-contracted, or
\item $\pi(|Z_{i}|)$ is an irreducible $k$-dimensional subvariety of $Y$.
\end{itemize}
Let $p_{i}: W_{i} \to T_{i}$ be the base change of $\pi$ to $\pi(Z_{i})$.  Applying Corollary \ref{cor:makingvertical} to $p_{i}$, we see that each $Z_{i}$ is $\mathbb{Q}$-rationally equivalent to a $\pi$-contracted cycle (as cycles on $W_{i}$, and hence after pushforward as cycles on $X$). Thus $Z$ is numerically equivalent to a sum of $\pi$-contracted cycles.
\qed

\begin{rmk}In a slightly different setting, \cite[Corollary 12.1.5.1]{KM92} shows
that if $\pi:X\to Y$ is a projective morphism with connected fibers between algebraic spaces with rational singularities, and $R^1\pi_*\mathcal O_X=0$, then the GK property holds for $\pi$ when $k=1$.
\end{rmk}

\begin{rmk}The condition $\dim_{\mathbb Q}(A_{0}(F)_{\mathbb Q}) =1$ can be relaxed, asking instead that the closure of the locus in $Y$ where it fails have dimension at most $k-1$.
\end{rmk}

\begin{cor}\label{cor:rccGK}
Suppose that $\pi: X \to Y$ is a morphism of projective varieties over $\mathbb{C}$ such that $Y$ is smooth and the general fiber of $\pi$ is rationally chain connected.  Then $\pi$ satisfies the GK property.
\end{cor}

\begin{proof}
By Proposition \ref{swconjforbirationalsmooth} and Lemmas \ref{GKcompose} and \ref{GKdescend}, we may replace $\pi$ by a birationally equivalent morphism.  Since $\pi$ is flat over an open subset, there is a rational map $Y \dashrightarrow \mathrm{Hilb}(X)$ induced by the fibers.  Since we may replace $Y$ by a birational model $Y'$ resolving this map (cf. Proposition \ref{swconjforbirationalsmooth}) and may replace $X$ by the main component of $X \times_{Y} Y'$ for $X$, we may assume that $\pi$ is flat (and $Y$ is still smooth).  Note that the general fiber is unchanged by the flattening operation.  By Theorem \ref{firstgkthrm} it suffices to show that rational chain connectedness is a flat deformation invariant.  It suffices to consider the case when we have a flat morphism $\pi: X \to C$ for a curve $C$ where the general fiber is RCC.  Let $F_{0}$ denote the special fiber; we must show any two points $x,x'$ of $F_{0}$ are connected by a chain of rational curves.  Choose a curve $C'$ on $X$ that dominates $C$ and contains $x$ and $x'$.  For a general fiber $F$, there is a rational chain in $F$ connecting all the points of $C' \cap F$.  Using the uncountability of $\mathbb{C}$, one sees that such rational chains form a family of curves which dominates the base.  By taking a closure and applying \cite[II.2.4 Corollary]{kollar96}, we obtain a chain of rational curves in $F_{0}$ which connects $x$ and $x'$.
\end{proof}

\begin{exmple} \label{surfaceproductexample}
Let $S$ be a smooth surface such that $A_{0}(S) = \mathbb{Z}$.  By the work of \cite{mumford68} and \cite{roitman72}, this implies that $p_{g} = 0$ and $\mathrm{Alb}(S)$ is trivial.  Examples include any rational surface $S$ and conjecturally any surface with $q=p_{g}=0$.

Suppose that $Y$ is another smooth surface.  We claim that there is an isomorphism
\begin{equation*}
N_{2}(S \times Y) \cong \mathbb{R} \oplus (N_{1}(S) \otimes N_{1}(Y)) \oplus \mathbb{R}.
\end{equation*}
For surfaces over $\mathbb{C}$, this follows easily from Hodge theory and the Kunneth formula.   (In fact, this argument also works for any surface $S$ satisfying $q = p_{g} = 0$.)  We present an alternative ``algebraic'' approach valid over any uncountable algebraically closed field.

Let $\pi: S \times Y \to Y$ be the projection map.  By Theorem \ref{firstgkthrm} the kernel of $\pi_{*}: N_{2}(S \times Y) \to \mathbb{R}$ is spanned by $\pi$-contracted irreducible surfaces.  First note that any irreducible surface contracted to a point by $\pi$ is a fiber.

Second, suppose that the image of a $\pi$-contracted surface $T$ is a curve $C$ on $Y$.  Let $C'$ be a normalization of $C$.  We show that $N_{2}(S \times C') \cong N_{1}(S) \times \mathbb{R}$.  If an effective surface $T'$ on $S \times C'$ does not dominate $S$, then it is the pullback of a divisor on $S$.  If it does dominate $S$, then it induces a morphism $S \to \mathrm{Jac}(C')$.  But by assumption on the Albanese map this morphism is trivial.  So after twisting by the pullback of a line bundle from $S$, the divisor $T'$ is the pullback of a divisor on $C'$.

Putting these together, we see $N_{2}$ is spanned by three kinds of classes: fibers of $\pi_{1}$, fibers of $\pi_{2}$, and products $\pi_{1}^{*}D \cdot \pi_{2}^{*}E$ for a divisor $D$ on $S$ and $E$ on $Y$. The natural map $\pi_{1}^{*}N_{1}(S) \otimes \pi_{2}^{*}N_{1}(Y) \to N_{2}(S \times Y)$ is easily seen to be injective.\qed
\end{exmple}

\nocite{*}
\bibliographystyle{amsalpha}
\bibliography{gkproperty}

\end{document}